\documentclass[letterpaper, 10 pt, conference]{ieeeconf}  

\IEEEoverridecommandlockouts                              

\usepackage{todonotes}
\usepackage{latexsym,amssymb,amsmath, amsbsy,amsopn, amstext}
\usepackage{graphicx,epsfig,tikz,pgf,hyperref,cancel,color,float,ifpdf}
\usepackage{extarrows,mathrsfs}
\usepackage{amsmath,  amssymb,stfloats}
\usepackage{tikz}
\usepackage{caption2}
\usetikzlibrary{automata, positioning, arrows}
\usepackage[ruled,vlined]{algorithm2e}

\usepackage{etoolbox}
\makeatletter
\patchcmd{\@makecaption}
  {\scshape}
  {}
  {}
  {}
\makeatother
\usepackage{makecell}

\graphicspath{{./figures/}}

\newtheorem{definition}{\bfseries Definition}
\newtheorem{proposition}{\bfseries Proposition}
\newtheorem{example}{\bfseries Example}
\newtheorem{theorem}{\bfseries Theorem}
\newtheorem{lemma}{\bfseries Lemma}

\newtheorem{remark}{\bfseries Remark}

\newtheorem{program}{\bfseries Program}

\def\A{\mathcal{A}}

\def\d{\delta}

\def\M{\mathcal{M}}

\newcommand{\R}{{\mathbb R}}

\newcommand{\Z}{{\mathbb Z}}

\allowdisplaybreaks[4]

\usepackage{xcolor}
\newif\ifdraft
\drafttrue

\title{\LARGE \bf
Finding Matrix Sequences with a High Asymptotic Growth Rate for Linear Constrained Switching Systems} 

\author{Yuhao Zhang, Xiangru Xu
\thanks{Yuhao Zhang and Xiangru Xu are with the Department of Mechanical Engineering, University of Wisconsin-Madison,
        Madison, WI, USA. Email: 
        {\tt\small \{yuhao.zhang2,xiangru.xu\}@wisc.edu}}%
}

\begin{document}
\maketitle


\begin{abstract}
Linear constrained switching systems are linear switched systems whose switching sequences are constrained by a deterministic finite automaton. This work investigates how to generate a sequence of matrices with an asymptotic growth rate  close to the constrained joint spectral radius (CJSR) for constrained switching systems, based on our previous result that reveals the equivalence of a constrained switching system and a lifted arbitrary switching system.
By using the dual solution of a sum-of-squares optimization program, an algorithm is designed for the lifted arbitrary switching system to produce a sequence of matrices with an asymptotic growth rate that is close to the CJSR of the original constrained switching system. It is also shown that a type of existing algorithms designed for arbitrary switching systems can be applied to the lifted system such that the desired sequence of matrices can be generated for the constrained switching system. Several numerical examples are provided to illustrate the better performance of the proposed algorithms compared with existing ones. 


\end{abstract}

\section{Introduction}\label{sec:intro}

A switched system is a dynamical system that consists of a set of subsystems and a logical rule that orchestrates switching between these subsystems \cite{lin2009stability,liberzon2012switching}. 
The discrete-time linear switched system associated with a finite set of matrices $\mathcal{A}=\{A_1, A_2, ..., A_m\} \subset \R^{n\times n}$ can be modeled as 
\begin{align}\label{eqswitch}
x_{k+1}=A_{\sigma_k}x_k
\end{align}
where $\sigma_k\in[m]:=\{1, 2, ..., m\}$ is the switching mode of the system and $x_k\in\R^n$ is the state at time $k \in \mathbb{Z}_{\geq 0}$. The system (\ref{eqswitch}) is called an \emph{arbitrary switching system} and denoted by $S(\mathcal{A})$ as there is no constraint on the switching sequence. The \emph{joint spectral radius} (JSR), which characterizes the maximal asymptotic growth rate of an infinite
product of matrices of the set $\mathcal{A}$, was firstly introduced in \cite{rota1960note} by Rota and Strang, 
and has been studied for decades because of its wide application in number theory, computer network,  wavelet functions, and signal processing \cite{jungers2009joint}. 
In particular, the value of JSR of \eqref{eqswitch} is related to the stability of  \eqref{eqswitch} since \eqref{eqswitch} is stable if and only if its JSR is smaller than 1 (see Corollary 1.1 in \cite{jungers2009joint}). There have been lots of algorithms (e.g., Gripenberg, conitope, and sum-of-squares (SOS)) proposed in the past decades  to approximate the value of JSR  (see, e.g., \cite{jungers2009joint}).

In this work, we consider the switched system \eqref{eqswitch} whose switching sequences are constrained by a deterministic finite automaton (DFA) $\mathcal{M}$. We refer to such a switched system as the \emph{constrained switching system}, denoted as $S(\mathcal{A},\mathcal{M})$. Similar to the concept of JSR that applies to the arbitrary switching system, the \emph{constrained joint spectral radius} (CJSR) characterizes the stability of the constrained switching system \cite{dai2012gel}. The approximation of CJSR, however,  is much more difficult than that of JSR. In \cite{philippe2016stability},  Philippe et al. propose a semi-definite programming based method
to approximate CJSR, where the T-product lift and the M-path dependent lift methods are used to improve the approximation accuracy; in \cite{xu2018approximation}, Xu and Behcet propose a unified matrix-based formulation for arbitrary and constrained switching systems and prove that the CJSR of a constrained switching system is equivalent to the JSR of a \emph{lifted} arbitrary switching system, such that the approximation of CJSR can be reduced to the approximation of JSR for which many off-the-shelf algorithms exist; see also \cite{wang2017stability,kozyakin2014berger} and references therein.

This work investigates how to generate a sequence of matrices with an asymptotic growth rate close to the CJSR based on the lift-based approach in \cite{xu2018approximation}. Finding such sequences is useful in various applications, such as providing a more accurate lower bound for the CJSR and testing the stability of linear switched systems \cite{legat2020certifying}. There are a few existing algorithms that can be utilized to produce the optimal asymptotic growth rate sequences for arbitrary switching systems: in \cite{gripenberg1996computing}, Gripenberg proposes a branch-and-bound algorithm that improves the search by a priori fixed absolute error; other branch-and-bound algorithms include the complex polytope algorithm in \cite{guglielmi2008algorithm} and the conitope method in \cite{jungers2014lifted}. All of the above algorithms can only deal with arbitrary switching systems.  For constrained switching systems, the only existing result known to us is \cite{legat2020certifying} (\cite{legat2016generating} is a preliminary version of \cite{legat2020certifying}), where Legat et al. propose to use the dual solution of a SOS program for generating the desired switching  sequences. 

The contribution of this work is at least twofold. First, we  propose a novel lift-based algorithm for finding a sequence of matrices with an  asymptotic growth rate close to the CJSR of linear constrained switching systems; 
we prove that the switching sequences generated by our algorithm are always accepted by the constraining DFA, and the asymptotic growth rate can be made arbitrarily close to the CJSR. 
Second, we prove that the lift-based approach proposed in \cite{xu2018approximation} enables one to leverage a type of existing algorithms that are designed for arbitrary switching systems to generate high-growth sequences for constrained switching systems. Additionally,  through several numerical examples we show that the lift-based algorithms can  generate switching sequences with a success rate higher than that in \cite{legat2020certifying}. The remainder of the paper is organized as follows. Section \ref{sec:stp} introduces the formal definitions of JSR and CJSR and the equivalence between them under the lift-based framework. Section \ref{sec:main} presents the sequence generation algorithm, the theoretical bound of matrix products generated by the algorithm, the extendability to other existing algorithms, and several numerical examples.  Section \ref{sec:conclusion} provides some concluding remarks.

\section{Preliminaries}\label{sec:stp}

\subsection{JSR \& CJSR}

Given a finite set of matrices $\mathcal{A}=\{A_1, A_2, ..., A_m\} \subset \R^{n\times n}$ and a switching sequence $\sigma=\sigma_1\dots\sigma_{k}$  with $\sigma_1,\dots,\sigma_{k}\in[m]$ where $[m]:=\{1, 2, ..., m\}$, we define 
\begin{align}
A_\sigma:=A_{\sigma_{k}}\dots A_{\sigma_1}.\label{sigorder}
\end{align}

\begin{definition}\cite{jungers2009joint}
The JSR of an arbitrary switching system $S(\mathcal{A})$ is defined as $\rho(\mathcal{A})=\limsup_{k\rightarrow\infty}\rho_k(\A)^{1/k}$ 
where $\rho_k(\A)=\max_{\sigma\in[m]^k}\|A_\sigma\|$, and $\|\cdot\|$ is any given sub-multiplicative matrix norm  on $\R^{n\times n}$.
\end{definition}


In this work, the deterministic finite automaton   (DFA) is used to represent the constraints on the switching sequences. 

\begin{definition}
The DFA  $\mathcal{M}$ is a 3-tuple $(Q,U,f)$ consisting of a finite set of states $Q=\{q_1,q_2,\dots,q_\ell\}$,
a finite set of input symbols $U=\{1,2,\dots,m\}$ and a transition function $f: Q\times U\rightarrow Q$. 
\end{definition}

Note that the function $f$ in $\mathcal{M}$ may not be defined for all state-input pairs. For system \eqref{eqswitch}, we say a finite switching sequence $\sigma=\sigma_1...\sigma_k$ is \emph{accepted} by $\M$ or $\M$-accepted if $\sigma_1,...,\sigma_k\in U$ and there exists a finite state sequence $q_{j_1}q_{j_2}\dots q_{j_{k+1}}$ such that $q_{j_1},q_{j_2},\dots, q_{j_{k+1}}\in Q$ and $q_{j_{i+1}}=f(q_{j_i},\sigma_i)$ are defined for $i=1,...,k$; an infinite switching sequence accepted by $\M$ is defined similarly by taking $k=\infty$ \cite{philippe2016stability,eilenberg1974automata}. Denote the set of switching sequences accepted by $\mathcal{M}$ as $L(\mathcal{M})$. 


\begin{definition}\label{dfn:STM}\cite{xu2012matrix}
Given a DFA $\mathcal{M}=(Q,U,f)$ where $Q=\{q_1,\dots,q_\ell\},\quad U=\{1,\dots,m\}$, its 
\emph{transition structure matrix} is defined as
\begin{align}
F=[F_1\;F_2\;\dots\;F_m]\in \mathbb{R}^{\ell\times m\ell}\label{TSMF}
\end{align}
where $F_j\in \mathbb{R}^{\ell\times \ell}$ is defined as follows: for $j\in[m]$,  
\begin{equation}\label{structmatrix}
{F_j}_{(s,t)}=\begin{cases}
1,\quad\mbox{if} \;q_s= f(q_t,j);\\
0,\quad\mbox{otherwise.}
\end{cases}
\end{equation}
\end{definition}

\begin{lemma}\label{cor1} (Corollary 1 in \cite{xu2018approximation})
	Given a switching sequence $\sigma=\sigma_0\dots\sigma_{k-1}\in[m]^k$, $\sigma\in L(\mathcal{M})$ if and only if  $F_{\sigma_{k-1}}\dots F_{\sigma_0}\neq{\bf 0}$.  
\end{lemma}

Formally, the \emph{constrained switching system} $S(\mathcal{A},\mathcal{M})$ is the linear switching system as shown in \eqref{eqswitch} where $A_i\in\mathcal{A}$ for $i\in[m]$ and the switching sequence $\sigma\in L(\mathcal{M})$ \cite{philippe2016stability}.

\begin{definition} \cite{philippe2016stability}   
The CJSR of a constrained switching system $S(\mathcal{A},\mathcal{M})$ is defined as 
$\rho(\mathcal{A}, \mathcal{M})=\limsup _{k \rightarrow \infty} \rho_{k}(\mathcal{A}, \mathcal{M})^{1 / k}$ 
where $\rho_{k}(\mathcal{A}, \mathcal{M})=\max _{\sigma \in[m]^{k}, \sigma \in L(\mathcal{M})}\left\|A_{\sigma}\right\|$,  
and $\|\cdot\|$ is any given sub-multiplicative matrix norm on $\R^{n\times n}$.
\end{definition}



	




\subsection{Equivalence between JSR and CJSR}

\begin{definition}\cite{horn1990matrix}
Given matrices $A=(a_{ij}) \in \mathbb{R}^{m \times n}$ and $B \in \mathbb{R}^{p \times q}$, the Kronecker product $A \otimes B$ is defined as
$$
A \otimes B:=\left(\begin{array}{ccc}
a_{11} B & \cdots & a_{1 n} B \\
\vdots & \ddots & \vdots \\
a_{m 1} B & \cdots & a_{m n} B
\end{array}\right).
$$
\end{definition}


The following two properties of Kronecker product will be used in later sections \cite{horn1990matrix,lancaster1972norms}:
\begin{itemize}
    \item Given matrices $A \in \mathbb{R}^{m_{A} \times n_{A}}, B \in \mathbb{R}^{m_{B} \times n_{B}}, C \in$
    $\mathbb{R}^{n_{A} \times n_{C}}, D \in \mathbb{R}^{n_{B} \times n_{D}},$ it holds that
    \begin{equation}\label{kron_prop1}
    (A \otimes B)(C \otimes D)=(A C) \otimes(B D)
    \end{equation}
    \item Given two matrices $A \in \mathbb{R}^{m_{A} \times n_{A}}$, $B \in \mathbb{R}^{m_{B} \times n_{B}}$ and any sub-multiplicative norm $\left\|\cdot\right\|$, it holds that
    \begin{equation}\label{kron_prop2}
    \|A \otimes B\|=\|A\|\|B\|
    \end{equation}
\end{itemize}

Given a system $S(\mathcal{A},\mathcal{M})$ where $\mathcal{A}=\{A_1, A_2, ..., A_m\} \subset \R^{n\times n}$ and the transition structure matrix, $F$, of $\mathcal{M}$ is defined in \eqref{TSMF}-\eqref{structmatrix}, we define a finite set of matrices
\begin{align}
\mathcal{A}_\M=\{\Phi_1,\dots,\Phi_m\}\label{AM}
\end{align}
where
	\begin{align}
	\Phi_i=F_i\otimes A_i\in \mathbb{R}^{n\ell\times n\ell},\;\forall i\in[m].\label{Phii}
	\end{align}

The arbitrary switching system $S(\mathcal{A}_\mathcal{M})$ can  be considered as a \emph{lifted system} of the constrained switching system $S(\mathcal{A},\mathcal{M})$. For more details on the relationship between $S(\mathcal{A}_\mathcal{M})$ and $S(\mathcal{A},\mathcal{M})$, we refer the reader to \cite{xu2018approximation}. The following result from \cite{xu2018approximation} reveals that the CJSR of $S(\mathcal{A},\mathcal{M})$ and the JSR of $S(\mathcal{A}_\mathcal{M})$ are equivalent.

\begin{lemma} (adopted from Theorem 2 in \cite{xu2018approximation})\label{cjsr=jsr_theorem}
The following equality holds: 
$\rho(\mathcal{A}, \mathcal{M})=\rho\left(\mathcal{A}_{\mathcal{M}}\right)$.
\end{lemma}

The importance of Lemma \ref{cjsr=jsr_theorem} lies in that it enables one to convert the problem of approximating the CJSR of $S(\mathcal{A},\mathcal{M})$  into the problem of approximating the JSR of its lifted system $S(\mathcal{A}_\M)$, for which many off-the-shelf algorithms exist. The simulation result in \cite{xu2018approximation} shows that the lift-based method can generate the CJSR with  significantly higher accuracy within a much shorter computational time.

\begin{example}\label{exp:0}
Consider the constrained switching system $S(\mathcal{A},\mathcal{M})$ given in  \cite{philippe2016stability} where  $\mathcal{A}=\{A_1,A_2,A_3,A_4\}$ with
$$
\begin{array}{ll}
A_{1}=\left(\begin{array}{cc}
0.94 & 0.56 \\
-0.35 & 0.73
\end{array}\right),\quad A_{2}=\left(\begin{array}{cc}
0.94 & 0.56 \\
0.14 & 0.73
\end{array}\right), \\
A_{3}=\left(\begin{array}{cc}
0.94 & 0.56 \\
-0.35 & 0.46
\end{array}\right),\quad A_{4}=\left(\begin{array}{cc}
0.94 & 0.56 \\
0.14 & 0.46
\end{array}\right),
\end{array}
$$
and the DFA $\mathcal{M}=(Q,U,f)$ is given by $Q=\{q_1,q_2,q_3,q_4\}$, $U=\{1,2,3,4\}$ with its transition map shown in Fig~\ref{fig:automaton1}. Using (\ref{structmatrix}) we can compute the transition structure matrix $F = [F_1\ F_2\ F_3\ F_4]$ where $F_{1}=[\delta_{4}^3,\delta_{4}^3,\delta_{4}^3,\delta_{4}^3]$, $F_{2}=[\delta_{4}^0,\delta_{4}^1,\delta_{4}^1,\delta_{4}^0]$, $F_{3}=[\delta_{4}^2,\delta_{4}^0,\delta_{4}^2,\delta_{4}^0]$, $F_{4}=[\delta_{4}^0,\delta_{4}^0,\delta_{4}^4,\delta_{4}^0]$. Here, $\d_n^k$ is the standard basis vector (i.e., a vector of dimension $n$ with 1 in the $k$th coordinate and 0's elsewhere for $k\in[n]$) and $\d_n^0$ is the zero vector of dimension $n$. We can calculate matrices $\Phi_{i}=F_{i} \otimes A_{i}$ for $i \in[4]$ and define the set of matrices $\mathcal{A}_{\mathcal{M}}=\left\{\Phi_{1}, \Phi_{2}, \Phi_{3}, \Phi_{4}\right\}$. 
Then by Lemma \ref{cjsr=jsr_theorem},  $\rho(\mathcal{A}, \mathcal{M})=\rho\left(\mathcal{A}_{\mathcal{M}}\right)$ holds.\hfill$\Box$
\begin{figure}[!ht]
    \centering
    \includegraphics[width=0.40\textwidth]{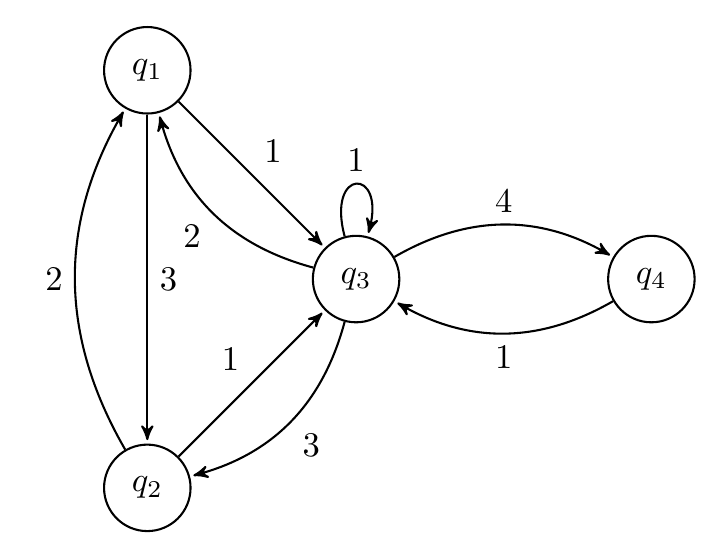}
    \caption{DFA $\mathcal{M}$ in Example~\ref{exp:0}.}
    \label{fig:automaton1}
\end{figure}
\end{example}

\section{Main Results}\label{sec:main}

\subsection{Dual SOS Program}
In this subsection, we will present a dual SOS program for the lifted arbitrary switching system $S(\mathcal{A}_\mathcal{M})$. Recall that the set $\mathcal{A}_\mathcal{M}$ is defined in \eqref{AM} with matrices $\Phi_i$ defined in \eqref{Phii}.


Based on Theorem 2.2 in \cite{parrilo2008approximation}, the following Program \ref{primal} provides a SOS-based algorithm for approximating $\rho(\mathcal{A}_\mathcal{M})$.

\begin{program}(Primal)\label{primal}
\begin{equation}
\begin{aligned}
 &\inf _{p(x) \in \mathbb{R}_{2 \mathrm{d}}[x],  \gamma \in \mathbb{R}}  \gamma \\
& \quad p(x) \text{ is  SOS}, \\
& \quad \gamma^{2 d} p(x)-p\left(\Phi_{i} x\right) \text{ is  SOS}, \quad \forall i \in [m],
\end{aligned}
\end{equation}
\end{program}
where 
$d$ is a fixed positive integer and $\mathbb{R}_{2 \mathrm{d}}[x]$ is the set of homogeneous polynomials of degree 2d \cite{parrilo2000structured}. 

Denote $\rho_{SOS,2d}(\mathcal{A}_\mathcal{M})$ as the solution of Program \ref{primal}. By  Theorem 2.2 in~\cite{parrilo2008approximation}, a feasible solution of Program 1 provides an
upper bound for $\rho(\mathcal{A}_\mathcal{M})$:
\begin{equation}\label{rho_sos}
    \rho(\mathcal{A}_\mathcal{M}) \leq \rho_{SOS,2d}(\mathcal{A}_\mathcal{M})
\end{equation}


The dual variables of Program \ref{primal} are linear functionals over homogeneous polynomials of degree $2d$. 
The dual of the feasibility version of Program \ref{primal} is given by the following Program \ref{dual} \cite{legat2020certifying}.
\begin{program}(Dual)\label{dual}
\begin{equation}\label{equ:dual}
\sum_{i=1}^{m} \widetilde{\mathbb{E}}_{i} \left[p\left(\Phi_{i} x\right)\right] \geq \gamma^{2 d} \sum_{i=1}^{m} \widetilde{\mathbb{E}}_{i}[p(x)], \quad \forall p(x) \in \Sigma_{\mathbf{2} \mathbf{d}},\\
\end{equation}
\begin{equation}\label{equ:dual2}
\begin{aligned}
\sum_{i=1}^{m} \widetilde{\mathbb{E}}_{i} & \left[\sum_{j=1}^{n} x_{j}^{2 d}\right]=1, \\
\widetilde{\mathbb{E}}_{i} & \in \Sigma_{\mathbf{2} \mathbf{d}}^{*},\quad \forall i \in [m],
\end{aligned}
\end{equation}
where $\Sigma_{\mathbf{2} \mathbf{d}}$ is the cone of homogeneous SOS polynomials of degree $2d$, $\Sigma_{\mathbf{2} \mathbf{d}}^{*}$ is the dual of $\Sigma_{\mathbf{2} \mathbf{d}}$, and $\widetilde{\mathbb{E}}_i$ is the pseudo-expectation \cite{barak2012hypercontractivity}. Here, the pseudo-expectation can be seen as the expectation of a pseudo-distribution $\tilde{\mu}$ on $\R^n$ and satisfies:
\begin{equation}
\langle\widetilde{\mathbb{E}}_i, p(x)\rangle=\widetilde{\mathbb{E}}_i[p(x)]=\int_{\mathbb{S}^{n-1}} p(x) \mathrm{d} \widetilde{\mu}
\end{equation}
where $\mathbb{S}^{n-1}$ is the $(n-1)$-dimensional sphere \cite{legat2016generating,barak2012hypercontractivity}. Since $p(x)$ is a homogeneous SOS polynomials of degree $2d$ and $\widetilde{\mathbb{E}}_i$ belongs to the dual of the SOS cone, $\widetilde{\mathbb{E}}_i[p(x)] \geq 0$ holds.
\end{program}
For any matrix $ \Phi_i \in \mathbb{R}^{n\ell \times n\ell}$, we define $\mathcal{L}_{\Phi_i}$ as the linear map such that $\mathcal{L}_{\Phi_i}(p(x)) = p(\Phi_i x)$ and define $\mathcal{L}_{\Phi_i}^*$ as the adjoint linear map such that $\langle\mathcal{L}_{\Phi_i}^* \widetilde{\mathbb{E}}_{i}, p(x)\rangle = \langle \widetilde{\mathbb{E}}_{i}, p(\Phi_i x)\rangle$ for all $\widetilde{\mathbb{E}}_{i} \in \Sigma_{\mathbf{2} \mathbf{d}}^{*}$ and $p(x) \in \Sigma_{\mathbf{2} \mathbf{d}}$. Therefore, the dual constraint (\ref{equ:dual}) is equivalent to:
\begin{equation}\label{equ:nop}
\sum_{i=1}^{m} \mathcal{L}_{\Phi_i}^* \widetilde{\mathbb{E}}_{i} \succeq \gamma^{2 d} \sum_{i=1}^{m} \widetilde{\mathbb{E}}_{i}
\end{equation}
Thus, with any positive integer $d$ and any $\gamma<\rho_{SOS,2d}(\mathcal{A}_\mathcal{M}),$ a set of dual variables $\{\widetilde{\mathbb{E}}_{1},\dots,\widetilde{\mathbb{E}}_{m}\}$ can be obtained by solving Program \ref{dual} without specifying $p(x)$.

\begin{remark}\label{jsr_toolbox}
In this work, Program \ref{primal} is mainly used to introduce its dual problem, Program \ref{dual}. We will use other algorithms, rather than Program \ref{primal}, to approximate $\rho(\mathcal{A}_\mathcal{M})$. For example, consider the  system $S(\mathcal{A},\mathcal{M})$ and its lifted system $S(\mathcal{A}_\mathcal{M})$ shown in Example \ref{exp:0}. By choosing the Gripenberg’s algorithm and the conitope algorithm in the \emph{jsr} function of the JSR toolbox in \cite{vankeerberghen2014jsr}, we obtain the following bounds in 8.9 seconds in a computer with 3.7 GHz CPU and 32GB memory:
\begin{equation} \label{equ:jsr_bounds}
0.97481720 \leq \rho(\mathcal{A}_\mathcal{M}) \leq 0.97481730.    
\end{equation}
If we choose $2d=4$  and use the \emph{jsr\_opti\_sos} function that employs Program \ref{primal}, then it takes about 108 seconds to obtain the following bounds:
\begin{align}
0.69673837 \leq \rho(\mathcal{A}_\mathcal{M}) \leq 0.98632317.\label{ex1eq2}
\end{align}
Clearly, the bounds shown in \eqref{equ:jsr_bounds} is more accurate than \eqref{ex1eq2} with a much shorter computational time. 
\end{remark}

\begin{remark}
When the degree $2d$ is chosen to be large, the SOS constraint corresponding to \eqref{equ:dual} involves a high order SOS polynomial, which renders the corresponding SDP difficult to solve. A possible method to alleviate this issue is to replace the dual constraint (\ref{equ:dual}) with the diagonally-dominant-sum-of-squares (DSOS) constraint or the scaled-diagonally-dominant-sum-of-squares (SDSOS) constraint  \cite{ahmadi2014dsos}. Nonetheless, the numerical examples in Section \ref{sec:example} will show that  Program \ref{dual} can normally generate satisfying results with small values of $d$.
\end{remark}

\subsection{Sequence Generating Algorithm}

In this subsection, we will present Algorithm \ref{algorithm2} based on the dual SOS program for  $S(\mathcal{A}_\mathcal{M})$. We will prove that the sequence of matrices $\Phi_{\sigma_{1}}, \cdots, \Phi_{\sigma_{k}}$
generated by Algorithm \ref{algorithm2} has an asymptotic growth rate $\lim_{k \rightarrow \infty}\|\Phi_{\sigma_{1}} \cdots \Phi_{\sigma_{k}}\|_{2}^{\frac{1}{k}}$ that can be made close to $\rho(\mathcal{A}_\mathcal{M})$, and the switching sequence $\sigma=\sigma_1\sigma_{2}...\sigma_k$ is  always accepted by the constraining automaton (i.e., $\sigma\in L(\mathcal{M})$). Based on that, we will prove that the corresponding sequence of matrices $A_{\sigma_{1}}, \cdots, A_{\sigma_{k}}$ for the original constrained switching system $S(\mathcal{A},\mathcal{M})$ has an asymptotic growth rate $\lim_{k \rightarrow \infty}\|A_{\sigma_{1}} \cdots A_{\sigma_{k}}\|_{2}^{\frac{1}{k}}$ that can be made arbitrarily close to $\rho(\mathcal{A},\mathcal{M})$. 

Given a set of matrices $\mathcal{A}_\mathcal{M}=\{\Phi_1,\dots,\Phi_m\}$ with $\Phi_i$ defined in \eqref{Phii} and a set of dual variables $\{\widetilde{\mathbb{E}}_{1},\dots,\widetilde{\mathbb{E}}_{m}\}$ that are obtained by solving Program \ref{dual}, Algorithm \ref{algorithm2} below generates a switching sequence $\sigma=\sigma_1\sigma_{2}...\sigma_k$ such that the product of matrices $\Phi_{\sigma_{1}} \cdots \Phi_{\sigma_{k}}$ has an asymptotic growth rate close to $\rho(\mathcal{A}_\mathcal{M})$. 
Algorithm \ref{algorithm2} takes a positive integer $h$ as an input parameter, which can be considered as the ``horizon'' of the algorithm. Instead of generating each switching mode $\sigma_i$ one by one, Algorithm \ref{algorithm2}  generates $h$ switching modes simultaneously at each iteration step. Therefore, the total length of sequences generated by the algorithm, $k$, will be a multiple of $h$.

\begin{algorithm}
\SetAlgoLined
\KwIn {a set of matrices $\{\Phi_1,\dots,\Phi_m\}$, a set of dual variables $\{\widetilde{\mathbb{E}}_{1},\dots,\widetilde{\mathbb{E}}_{m}\}$, three positive integers $h,k,d$ where $k$ is a multiple of $h$}
\KwOut {a switching sequence $\sigma=\sigma_1 \sigma_{2}...\sigma_k$}
Choose an arbitrary polynomial $p_{0}(x) \in \operatorname{int}(\Sigma_{2 \mathrm{d}})$ \\
Set $p_{1}(x) \leftarrow p_{0}(x)$ \\
\For {$i=1,h+1,2h+1, \ldots,k-h+1$}{
\begin{equation*}
\begin{aligned}
& \sigma_{i},...,\sigma_{i+h-1} \leftarrow \arg \max _{\hat{\sigma}_1,...,\hat{\sigma}_h} \widetilde{\mathbb{E}}_{\hat{\sigma}_h}\left[p_{i}\left(\Phi_{\hat{\sigma}_1} \cdots \Phi_{\hat{\sigma}_h} x\right)\right]\\
& p_{i+h}(x) \leftarrow p_{i}\left(\Phi_{\sigma_{i}} \cdots \Phi_{\sigma_{i+h-1}} x\right)
\end{aligned}
\end{equation*}
}
 \caption{Generate a  sequence of matrices $\Phi_{\sigma_{1}}, \cdots, \Phi_{\sigma_{k}}$ with an asymptotic growth rate close to $\rho(\mathcal{A}_\mathcal{M})$}
 \label{algorithm2}
\end{algorithm}

To start with, we choose an arbitrary polynomial $p_0(x)$ in the interior of the cone of SOS homogeneous polynomials of degree $2d$, i.e. $p_0(x) \in \operatorname{int}(\Sigma_{2 \mathrm{d}})$. 
Given the set of matrices $\{\Phi_1,\dots,\Phi_m\}$ and the set of dual variables $\{\widetilde{\mathbb{E}}_{1},\dots,\widetilde{\mathbb{E}}_{m}\}$ computed from Program \ref{dual}, we define 
\begin{equation}\label{theta}
\theta_{k} := \widetilde{\mathbb{E}}_{\sigma_{k}}\left[p_0\left(\Phi_{\sigma_{1}} \cdots \Phi_{\sigma_{k}} x\right)\right].
\end{equation}
Then, the ``for loop'' generates a switching sequence 
such that $\theta_{k}$ 
remains large for increasing $k$. 
Algorithm \ref{algorithm2} terminates with $i=k-h+1$, from which we obtain 
\begin{align}
p_{k+1}(x)=p_0(\Phi_{\sigma_{1}} \cdots \Phi_{\sigma_{k}} x).\label{pk}
\end{align}

Note that the order of mode subscripts in the matrix product of Algorithm \ref{algorithm2}, and  \eqref{theta}, \eqref{pk} as well, is reversed from that of \eqref{sigorder}. Also note that due to the randomness of $p_0(x)$, the output of Algorithm \ref{algorithm2} may vary among different runs.

The following lemma provides an inequality on $\theta_{k}$ using the dual constraint (\ref{equ:dual}).


\begin{lemma}\label{decrease2}
Consider a finite set of matrices $\mathcal{A}_\mathcal{M}=\{\Phi_1,\dots,\Phi_m\}$. For any $d,h \in \Z_{>0}$, any $\gamma<\rho_{SOS,2d}(\mathcal{A}_\mathcal{M})$, and any $k\in \Z_{>0}$ that is a multiple of $h$, Algorithm \ref{algorithm2} produces a sequence $\sigma = \sigma_1\sigma_2...\sigma_k$, for which the sequence of $\theta_k$ defined in \eqref{theta} satisfies the following inequality:
$$
\theta_{k} \geq \frac{\gamma^{2d h}}{m^h} \theta_{k-h}
$$
\end{lemma}
\begin{proof}
Since $p_{0}(x)$ is SOS, 
$p_{k-h+1}\left(x\right)$ is also SOS.
Using (\ref{equ:dual}), we have
$$
\begin{aligned}
& \sum_{\hat{\sigma}_1...\hat{\sigma}_h \in [m]^h} \hat{\mathbb{E}}_{\hat{\sigma}_h}\left[p_{k-h+1}\left(\Phi_{\hat{\sigma}_1} \cdots \Phi_{\hat{\sigma}_h} x\right)\right] \\
&\quad \geq \gamma^{2 d} \sum_{\hat{\sigma}_1...\hat{\sigma}_{h-1} \in [m]^{h-1}} \widetilde{\mathbb{E}}_{\hat{\sigma}_{h-1}}\left[p_{k-h+1}\left(\Phi_{\hat{\sigma}_{1}} \cdots \Phi_{\hat{\sigma}_{h-1}} x\right)\right] \\
&\quad \vdots \\
&\quad \geq \gamma^{2 d h} \theta_{k-h}
\end{aligned}
$$
Since the left-hand side expression has $m^{h}$ positive terms and Algorithm \ref{algorithm2} picks the term with the highest value as $\theta_k$, we have $m^h\theta_{k} \geq \gamma^{2d h} \theta_{k-h}$. The conclusion follows immediately.
\end{proof}

\begin{lemma}\label{beta}(Lemma 6 in \cite{legat2016generating})
For any polynomial $p(x) \in \Sigma_{2 \mathrm{d}}$ and any matrix $\Phi \in \mathbb{R}^{n\ell \times n\ell}$, there exists a positive constant $\beta$ that does not depend on $\Phi$ such that
\begin{equation}
\beta\|\Phi\|_{2}^{2 d} p(x)-p(\Phi x) \text{ is SOS.}
\end{equation}
\end{lemma}

Proposition \ref{norm2} provides an estimate on the performance of the asymptotic  growth  rate  of matrices $\Phi_{\sigma_1},\cdots,\Phi_{\sigma_k}$ that are generated by  Algorithm \ref{algorithm2}.

\begin{proposition}\label{norm2} 
Consider a finite set of matrices $\mathcal{A}_\mathcal{M}=\{\Phi_1,\dots,\Phi_m\}$. 
For any $d,h \in \Z_{>0}$, any $\gamma<\rho_{SOS,2d}(\mathcal{A}_\mathcal{M})$, and any $k\in \Z_{>0}$ that is a multiple of $h$,  Algorithm \ref{algorithm2} produces a switching sequence  $\sigma=\sigma_1 \sigma_{2}...\sigma_k$ such that the following inequality holds: 
\begin{equation}\label{ineqprop}
\lim _{k \rightarrow \infty}\left\|\Phi_{\sigma_{1}} \cdots \Phi_{\sigma_{k}}\right\|_{2}^{\frac{1}{k}} \geq \frac{\gamma}{m^{\frac{1}{2d}}}
\end{equation}
\end{proposition}

\begin{proof}
Since there are finite switching modes (i.e., $\sigma_i \in [m], i=1,2,...$), there must be a mode $\bar{\sigma} \in [m]$  that appears infinitely many times in the sequence at multiples of $h$. Let $k_{1}$ be the smallest multiple of $h$ such that $\sigma_{k_{1}}=\bar{\sigma}$
and let $g(x)=p_{k_{1}+1}(x)$. Since $p_{0}(x) \in \operatorname{int}(\Sigma_{2 \mathrm{d}}),$ we know that $g(x) \in \operatorname{int}(\Sigma_{2 \mathrm{d}})$.

For an arbitrarily large integer $K,$ there exists a $k \geq K$ and $k$ is a multiple of $h$ such that $\sigma_{k_{1}+k}=\bar{\sigma}$. Let
$s_{k}=\left(\sigma_{k_{1}+1}, \ldots, \sigma_{k_{1}+k}\right) .$ By Lemma \ref{decrease2}, we have
\begin{equation}\label{norm_gur2}
\widetilde{\mathbb{E}}_{\bar{\sigma}}\left[g\left(\Phi_{s_{k}} x\right)\right] \geq \frac{\gamma^{2 d k}}{m^{k}} \widetilde{\mathbb{E}}_{\bar{\sigma}}[g(x)]
\end{equation}
where $\Phi_{s_{k}}:= \Phi_{\sigma_{k_{1}+1}}\cdots\Phi_{\sigma_{k_{1}+k}}.$
By Lemma \ref{beta}, there exists a constant $\beta$ that does not depend on $\Phi_{s_{k}}$ such that
\begin{equation}
\beta\left\|\Phi_{s_{k}}\right\|_{2}^{2d} g(x)-g\left(\Phi_{s_{k}} x\right) \text { is SOS. }
\end{equation}
Therefore,
\begin{equation}
\widetilde{\mathbb{E}}_{\bar{\sigma}}\left[\beta\left\|\Phi_{s_{k}}\right\|_{2}^{2d} g(x)\right] \geq \widetilde{\mathbb{E}}_{\bar{\sigma}}\left[g\left(\Phi_{s_{k}} x\right)\right]
\end{equation}
and by (\ref{norm_gur2}),
\begin{equation}\label{eqbetaineq}
\beta\left\|\Phi_{s_{k}}\right\|_{2}^{2 d} \widetilde{\mathbb{E}}_{\bar{\sigma}}[g(x)] \geq \frac{\gamma^{2 d k}}{m^{k}} \widetilde{\mathbb{E}}_{\bar{\sigma}}[g(x)]
\end{equation}
Since $\widetilde{\mathbb{E}}_{\bar{\sigma}}[g(x)]>0$, we divide both sides of \eqref{eqbetaineq} by $\widetilde{\mathbb{E}}_{\bar{\sigma}}[g(x)]$ and  get
\begin{equation}
\beta\left\|\Phi_{s_{k}}\right\|_{2}^{2 d} \geq \frac{\gamma^{2 d k}}{m^{k}}
\end{equation}
or equivalently,
\begin{equation}
\left\|\Phi_{s_{k}}\right\|_{2}^{\frac{1}{k}} \geq \frac{\gamma}{\left[\beta^{\frac{1}{k}} m\right]^{\frac{1}{2 d}}}
\end{equation}
Taking the limit of $K$ to $\infty$, we have
\begin{equation}
\lim _{k \rightarrow \infty}\left\|\Phi_{\sigma_{k_1+1}} \cdots \Phi_{\sigma_{k_1+k}}\right\|_{2}^{\frac{1}{k}} \geq \frac{\gamma}{m^{\frac{1}{2d}}}.
\end{equation}
Using the fact that $\lim _{k \rightarrow \infty}\left\|\Phi_{\sigma_{1}} \cdots \Phi_{\sigma_{k_1}}\right\|_{2}^{\frac{1}{k}} = 1$, the inequality (\ref{ineqprop}) follows immediately.
\end{proof}

The following result shows that the switching sequence generated by Algorithm \ref{algorithm2} is accepted by the DFA $\mathcal{M}$.
\begin{proposition}
Any switching sequence $\sigma=\sigma_1...\sigma_k$  produced by Algorithm \ref{algorithm2} is  accepted by $\mathcal{M}$, i.e., $\sigma \in L(\mathcal{M})$.
\end{proposition}

\begin{proof}
We prove the result by contradiction. 
Assume that the sequence $\sigma=\sigma_1...\sigma_k$ produced by Algorithm \ref{algorithm2} can not be accepted by $\mathcal{M}$, i.e. $\sigma \notin L(\mathcal{M})$. Then, from Lemma \ref{cor1}, we know $F_{\sigma_{1}}\dots F_{\sigma_k} = {\bf 0}$.
Recall that $\Phi_{\sigma_i}=F_{\sigma_i}\otimes A_{\sigma_i}, i=1,...k$, as shown in \eqref{Phii}.
Using the property in (\ref{kron_prop1}), we have
\begin{equation*}
\begin{aligned}
\Phi_{\sigma_{1}} \cdots \Phi_{\sigma_{k}} &=\left(F_{\sigma_{1}} \otimes A_{\sigma_{1}}\right)\left(F_{\sigma_{2}} \otimes A_{\sigma_{2}}\right) \ldots\left(F_{\sigma_{k}} \otimes A_{\sigma_{k}}\right) \\
&=\left(F_{\sigma_{1}} F_{\sigma_{2}} \ldots F_{\sigma_{k}}\right) \otimes\left(A_{\sigma_{1}} A_{\sigma_{2}} \cdots A_{\sigma_{k}}\right). \\
\end{aligned}
\end{equation*}
Taking the norm $\left\|\cdot\right\|_2$ on both sides of the equality above and using the property shown in (\ref{kron_prop2}), we have
\begin{align*}
\left\|\Phi_{\sigma_{1}} \cdots \Phi_{\sigma_{k}}\right\|_{2} &= \left\|\left(F_{\sigma_{1}} \ldots F_{\sigma_{k}}\right) \otimes\left(A_{\sigma_{1}} \cdots A_{\sigma_{k}}\right)\right\|_2\\
&= \left\|F_{\sigma_{1}} \ldots F_{\sigma_{k}}\right\|_2 \left\|A_{\sigma_{1}} \cdots A_{\sigma_{k}}\right\|_2.
\end{align*}
Taking $k$th root on both sides, and letting $k$ approach infinity, we have 
\begin{equation}\label{lim_phi}
\begin{aligned}
\lim _{k \rightarrow \infty} & \left\|\Phi_{\sigma_{1}} \cdots \Phi_{\sigma_{k}}\right\|_{2}^{\frac{1}{k}} \\
& = \lim _{k \rightarrow \infty}\left\|F_{\sigma_{1}} \ldots F_{\sigma_{k}}\right\|_2^{\frac{1}{k}} \lim _{k \rightarrow \infty}\left\|A_{\sigma_{1}} \cdots A_{\sigma_{k}}\right\|_2^{\frac{1}{k}}.
\end{aligned}
\end{equation}
Since $F_{\sigma_{1}}\dots F_{\sigma_k} = {\bf 0}$,
\begin{equation*}
\lim _{k \rightarrow \infty} \left\|\Phi_{\sigma_{1}} \cdots \Phi_{\sigma_{k}}\right\|_{2}^{\frac{1}{k}} = 0.
\end{equation*}
However, from Proposition \ref{norm2}, we have
\begin{equation*}
\lim _{k \rightarrow \infty}\left\|\Phi_{\sigma_{1}} \cdots \Phi_{\sigma_{k}}\right\|_{2}^{\frac{1}{k}} \geq \frac{\gamma}{m^{\frac{1}{2d}}} >0.
\end{equation*}
Therefore, the initial assumption $\sigma \notin L(\mathcal{M})$ is false, which completes the proof. 
\end{proof}

The following theorem is the main result of this paper. It shows that 
the sequence of matrices $A_{\sigma_{1}}, \cdots, A_{\sigma_{k}}$ with $\sigma=\sigma_1...\sigma_k \in L(\mathcal{M})$ generated by Algorithm \ref{algorithm2} has an asymptotic growth rate $\lim_{k \rightarrow \infty}\|A_{\sigma_{1}} \cdots A_{\sigma_{k}}\|_{2}^{\frac{1}{k}}$ that can be made close to any $\gamma<\rho_{SOS,2d}(\mathcal{A}_\mathcal{M})$. And since $\lim_{d \rightarrow \infty}\rho_{SOS,2d}(\mathcal{A}_\mathcal{M}) = \rho(\mathcal{A}_\mathcal{M}) = \rho(\mathcal{A},\mathcal{M})$, the asymptotic growth rate can be made arbitrarily close to $\rho(\mathcal{A},\mathcal{M})$. 

\begin{theorem}\label{new_theorem}
Consider a  constrained switching system $S(\mathcal{A},\mathcal{M})$ where $\mathcal{A}=\{A_1, A_2, ..., A_m\} \subset \R^{n\times n}$ and $\mathcal{M}$ is a DFA. For any $d,h \in \Z_{>0}$, any $\gamma<\rho_{SOS,2d}(\mathcal{A}_\mathcal{M})$, and any $k\in \Z_{>0}$ that is a multiple of $h$,  Algorithm \ref{algorithm2} produces a switching sequence  $\sigma=\sigma_1 \sigma_{2}...\sigma_k$ such that the following inequality holds:  
\begin{align}
\lim_{k \rightarrow \infty}\left\|A_{\sigma_{1}} \cdots A_{\sigma_{k}}\right\|_{2}^{\frac{1}{k}} \geq \frac{\gamma}{m^{\frac{1}{2d}}}.
\end{align}
\end{theorem}
\begin{proof}
By the definition of transition structure matrix, each column of $F_{\sigma_i}$ has at most one ``$1$'' with all other elements being ``$0$''. Therefore, $\left\|F_{\sigma_{1}} \ldots F_{\sigma_{k}}\right\|_2 \leq \sqrt{\ell}$,
where $\ell$ is the number of states in the DFA $\mathcal{M}$. It implies that
\begin{equation*}
\lim _{k \rightarrow \infty}\left\|F_{\sigma_{1}} \ldots F_{\sigma_{k}}\right\|_2^{\frac{1}{k}} = 1
\end{equation*}
Therefore, by (\ref{lim_phi}), we have
\begin{equation*}
\lim _{k \rightarrow \infty}\left\|A_{\sigma_{1}} \cdots A_{\sigma_{k}}\right\|_2^{\frac{1}{k}} = \lim _{k \rightarrow \infty} \left\|\Phi_{\sigma_{1}} \cdots \Phi_{\sigma_{k}}\right\|_{2}^{\frac{1}{k}} \geq \frac{\gamma}{m^{\frac{1}{2d}}}
\end{equation*}
where the last inequality is from Proposition \ref{norm2}. This completes the proof.
\end{proof}


In practice, switching sequences generated by Algorithm \ref{algorithm2} are periodic after some steps. Therefore, instead of generating sequences of the maximum length  by using Algorithm \ref{algorithm2}, we can compute the average spectral radius for all cycles with length smaller than a given length, and choose the cycle corresponding to the largest average spectral radius. Formally, 
$
\rho_{T}(\mathcal{A},\mathcal{M}) \leq \rho(\mathcal{A},\mathcal{M})
$
where 
$
\rho_{T}(\mathcal{A},\mathcal{M}) = \text{max}\{\rho(A_{c_T}...A_{c_1})^{1/T}: c=c_1 c_2\dots c_T \in L(\mathcal{M}), c \text{ is a cycle}\}.
$
The spectral radius of matrices corresponding to the cycle $c$ provides a lower bound for $\rho(\mathcal{A},\mathcal{M})$. 

\subsection{Leveraging Other Existing Algorithms}\label{other_algorithms}


Algorithm \ref{algorithm2}, Proposition \ref{norm2}, and Theorem \ref{new_theorem} are all based on the lift-based framework that establishes the equivalence between $S(\mathcal{A},{\mathcal{M}})$ and  $S(\mathcal{A}_{\mathcal{M}})$. In this subsection, we will show that this equivalence also allows us to leverage some existing algorithms for arbitrary switching systems to generate high-growth sequences for constrained switching systems. Numerical examples will be shown in subsection  \ref{sec:example}.

We summarize the Gripenberg algorithm (see, e.g.,  \cite{gripenberg1996computing}) that applies to the lifted system $S(\mathcal{A}_{\mathcal{M}})$ in Algorithm \ref{grip}, as an example to show that our lift-based method is not intended for a specific algorithm but is applicable for a type of algorithms. A theoretical result on the asymptotic growth rate similar to Theorem \ref{new_theorem} can be also obtained, which is omitted for the space limitation.  Proposition \ref{grip_prop} shows that the sequences generated by Algorithm \ref{grip} are $\mathcal{M}$-accepted.  

Recall that ${\mathcal{A}_\mathcal{M}} =  \{\Phi_1,\dots,\Phi_m\} \subset \mathbb{R}^{n\ell\times n\ell}$ where $\Phi_i$ is defined in (\ref{Phii}). The $n$-ary Cartesian power of the set ${\mathcal{A}_\mathcal{M}}$ is defined as 
$$
{\mathcal{A}_\mathcal{M}^n} = \underbrace{\mathcal{A}_\mathcal{M} \times \mathcal{A}_\mathcal{M} \times \cdots \times \mathcal{A}_\mathcal{M}}_{n}.
$$
Given a $t$-tuple $X = ( \Phi_{\sigma_1},\Phi_{\sigma_2},...,\Phi_{\sigma_t} )\in {\mathcal{A}_\mathcal{M}^t}$, we denote $\Pi(X) = \prod_{i = 1}^{t} \Phi_{\sigma_i}$ and $d(X) = \min_{1\leq j\leq t}\|\prod_{i=1}^j \Phi_{\sigma_i}\|^{1/j}$.

\setlength{\topmargin}{-20pt}
\begin{algorithm}
\SetAlgoLined
\KwIn {a set of matrices ${\mathcal{A}_\mathcal{M}} = \{\Phi_1,\dots,\Phi_m\}$, a non-negative scalar  $\epsilon \geq 0$ and a positive integer $t$}
\KwOut {a $k$-tuple of matrices ($k\leq t$): $\varphi_t = (\Phi_{\sigma_{1}}, \cdots, \Phi_{\sigma_{k}})$}
Set
$
T_1 = {\mathcal{A}_\mathcal{M}},\ \alpha_1 = \max _{\Phi \in {\mathcal{A}_\mathcal{M}}}\rho(\Phi)
$

Set $\varphi_1 = \text{arg}\max _{\Phi \in {\mathcal{A}_\mathcal{M}}}\rho(\Phi) $

\For {$i=2,3, \ldots,t$}{
$
T_i = \{(X,\Phi) \in {\mathcal{A}_\mathcal{M}^i}\ |\ X\in T_{i-1},\ \Phi \in {\mathcal{A}_\mathcal{M}},$\\
$\quad\quad\quad d(X,\Phi) > \alpha_{i-1} + \epsilon \}$\\
$\alpha_i = \max\{\alpha_{i-1},\ \max_{Y\in T_i} \rho(\Pi(Y))^{1/i}\}
$

\eIf{$\alpha_{i-1} \geq \max_{Y\in T_i} \rho(\Pi(Y))^{1/i}$}{
$\varphi_i =\varphi_{i-1}$}{
$\varphi_i = \text{arg} \max_{Y\in T_i}\rho(\Pi(Y))^{1/i} $}

}
 \caption{(Gripenberg Algorithm \cite{gripenberg1996computing}) Generate a  sequence of matrices $\Phi_{\sigma_{1}}, \cdots, \Phi_{\sigma_{k}}$ with a high growth rate.}
 \label{grip}
\end{algorithm}
\begin{proposition}\label{grip_prop}
Consider a set of matrices ${\mathcal{A}_\mathcal{M}} = \{\Phi_1,\dots,\Phi_m\}$ such that $\max _{\Phi \in {\mathcal{A}_\mathcal{M}}}\rho(\Phi) >0$. The sequence $\sigma=\sigma_1...\sigma_k$ corresponding to the $k$-tuple of matrices $(\Phi_{\sigma_{1}}, \cdots, \Phi_{\sigma_{k}})$ produced by Algorithm \ref{grip}  is  accepted by the DFA $\mathcal{M}$, i.e., $\sigma \in L(\mathcal{M})$.
\end{proposition}
\begin{proof}
We prove the result by contradiction. 
Assume that the sequence $\sigma=\sigma_1...\sigma_k$ corresponding to the $k$-tuple of matrices $(\Phi_{\sigma_{1}}, \cdots, \Phi_{\sigma_{k}})$ produced by Algorithm \ref{grip} can not be accepted by $\mathcal{M}$, i.e. $\sigma \notin L(\mathcal{M})$. Then, from Lemma \ref{cor1}, we know $F_{\sigma_{1}}\dots F_{\sigma_k} = {\bf 0}$.

Using the property in (\ref{kron_prop1}), we have
\begin{equation*}
\begin{aligned}
\Phi_{\sigma_{1}} \cdots \Phi_{\sigma_{k}} &=\left(F_{\sigma_{1}} \otimes A_{\sigma_{1}}\right)\left(F_{\sigma_{2}} \otimes A_{\sigma_{2}}\right) \ldots\left(F_{\sigma_{k}} \otimes A_{\sigma_{k}}\right) \\
&=\left(F_{\sigma_{1}} F_{\sigma_{2}} \ldots F_{\sigma_{k}}\right) \otimes\left(A_{\sigma_{1}} A_{\sigma_{2}} \cdots A_{\sigma_{k}}\right) \\
&= \bf{0}.
\end{aligned}
\end{equation*}
Since the output of Algorithm \ref{grip} is $(\Phi_{\sigma_1},...,\Phi_{\sigma_k})$, we have $\varphi_k = \text{arg} \max_{Y\in T_k}\rho(\Pi(Y))^{1/k} = (\Phi_{\sigma_1},...,\Phi_{\sigma_k})$, therefore,
\begin{equation}\label{alpha}
    \alpha_k =  \max_{Y\in T_k}\rho(\Pi(Y))^{1/k} = \rho(\Phi_{\sigma_1}\cdots \Phi_{\sigma_k})^{1/k} = 0.
\end{equation}
Thus,
$\alpha_1 \leq  \alpha_2 \leq ... \leq \alpha_{k} = 0.$ 
However, 
$
\alpha_1 = \max _{\Phi \in P}\rho(\Phi) >0.
$ 
Therefore, the initial assumption $\sigma \not\in L(\mathcal{M})$ is false, which completes the proof.
\end{proof}

Proposition \ref{grip_prop} shows that the sequence generated by Algorithm \ref{grip} for the lifted arbitrary switching system $S(\mathcal{A}_\mathcal{M})$ must also be an $\mathcal{M}$-accepted sequence for the original constrained switching system $S(\mathcal{A},\mathcal{M})$. 

\begin{remark}
Note that the finite sequence of matrices $\Phi_{\sigma_{1}}, \cdots, \Phi_{\sigma_{k}}$ produced by the Conitope algorithm also satisfies $\rho(\Phi_{\sigma_1}...\Phi_{\sigma_k}) > 0$ \cite{jungers2014lifted}. Therefore, a similar conclusion can be drawn that the switching sequence $\sigma=\sigma_1...\sigma_k$ produced by the Conitope algorithm on the lifted arbitrary switching system $S(\mathcal{A}_\mathcal{M})$ is accepted by the DFA $\mathcal{M}$.

Comparing Algorithm \ref{algorithm2} and Algorithm \ref{grip}, a common condition making them applicable for our lift-based method is that they both define a non-decreasing term for increasing lengths of generated sequences (e.g. $\theta_k$ in (\ref{theta}) and $\alpha_k$ in (\ref{alpha})). Therefore, besides the Gripenberg Algorithm and the Conitope Algorithm mentioned above, our lift-based method can be naturally extended to a type of existing algorithms designed for arbitrary switching systems as long as they can define a similar term which remains large for increasing growth rate of the generated sequences. 
\end{remark}

\subsection{Numerical Examples}\label{sec:example}

In \cite{legat2020certifying}, an algorithm is proposed for generating a sequence of matrices with a high asymptotic growth rate based on a dual SOS program  for $S(\mathcal{A},\mathcal{M})$. In this subsection, we will use several numerical examples to illustrate the effectiveness of the lift-based method proposed in this work and its comparisons with the method in  \cite{legat2020certifying}.

\begin{example}\label{exp_1}
Consider the constrained switching system $S(\mathcal{A},\mathcal{M})$ where the set $\mathcal{A}$ and the DFA $\mathcal{M}$ are given in Example \ref{exp:0}. We apply Algorithm \ref{algorithm2} to the lifted system $S(\mathcal{A}_{\mathcal{M}})$ to generate a switching sequence of a given length and look through all the circles in the sequence. The optimal cycle found depends on the initial choice of $p_0(x)$, but most of the time Algorithm \ref{algorithm2} with $2d = 2,h=3$ finds the following $\mathcal{M}$-accepted cycle:
\begin{equation}\label{cycle1}
1,1,2,1,2,3,1,1 \text{\quad or\quad} 2,1,2,3,1,1,1,1
\end{equation}
whose 8th roots of the corresponding spectral radius are both $0.97481720$. Note that this value is equal to the lower bound of $\rho(\mathcal{A}_\mathcal{M})$ given in (\ref{equ:jsr_bounds}). Also note that the two cycles shown in \eqref{cycle1} are essentially the same since the first cycle coincides with the second one if the beginning labels $``1,1"$ of the first cycle are moved to the end.

%
%

Example 3.18 in \cite{legat2020certifying} considers the same constrained switching system  $S(\mathcal{A},\mathcal{M})$ as above. By using a dual SOS program for $S(\mathcal{A},\mathcal{M})$, the following $\mathcal{M}$-accepted cycle is produced:
\begin{equation}\label{cycle20}
(3,1,2),(1,3,1),(3,1,2),(1,2,3),(2,3,1),(3,3,1)^{3}
\end{equation}
where the triplet $(u,v,w)$ denotes the edge between node $u$ and node $v$ with label $w$ in the automaton $\mathcal{M}$, and $``3"$ in the exponent means that the edges is taken $3$ times. 

The spectral radius corresponding to the cycle \eqref{cycle20} is the same as that of cycle \eqref{cycle1}. 
However, the algorithm in \cite{legat2020certifying} only produces one unique path that is accepted by $\mathcal{M}$ since all the starting nodes and ending nodes have been specified; in comparison, Algorithm \ref{algorithm2} in this work generates \emph{a set of} $\mathcal{M}$-accepted paths that have the same order of edge labels as Algorithm \ref{algorithm2} only specifies the labels of edges - this salient feature is because the constrained switching system is lifted into an associated arbitrary switching system. 
For example, the cycle \eqref{cycle1} may correspond to other path of the automaton $\mathcal{M}$ such as 
$
(2,1,2),(1,3,1),(3,1,2),(1,2,3),(2,3,1),(3,3,1)^{3}.
$
\hfill$\Box$
\end{example}



\begin{example}\label{exp:2}



Consider a constrained switching system $S(\mathcal{A},\mathcal{M})$ where the set $\mathcal{A}=\{A_1,A_2,A_3,A_4\}$ is given by
$$
\begin{array}{ll}
A_{1}=\left(\begin{array}{cc}
\ 0.55\ &  -0.69\ \\
\ 0.43\ & \ 0.25\
\end{array}\right), A_{2}=\left(\begin{array}{cc}
0.77 & 0.41 \\
-0.28 & 0.31
\end{array}\right), \\
A_{3}=\left(\begin{array}{cc}
-0.86 & -0.63 \\
-0.95 & -0.79
\end{array}\right), A_{4}=\left(\begin{array}{cc}
0.16 & 0.44 \\
-0.14 & 0.55
\end{array}\right),
\end{array}
$$
and the DFA $\mathcal{M}=(Q,U,f)$ with $Q=\{q_1,q_2,q_3,q_4\}$, $U=\{1,2,3,4\}$ and its transition map shown in Fig~\ref{fig:automaton2}. Note that this DFA is not \emph{strongly connected} since there is no path from states $q_2$ and $q_3$ to either $q_1$ or $q_4$. 
Using the JSR toolbox in \cite{vankeerberghen2014jsr}, we obtain the following bounds for $\rho(\mathcal{A}, \mathcal{M})$:
\begin{equation} \label{equ:jsr_bounds2}
0.84135421 \leq \rho(\mathcal{A},\mathcal{M}) \leq 0.84135429.
\end{equation}

\begin{figure}[!t]
    \centering
    \includegraphics[width=0.40\textwidth]{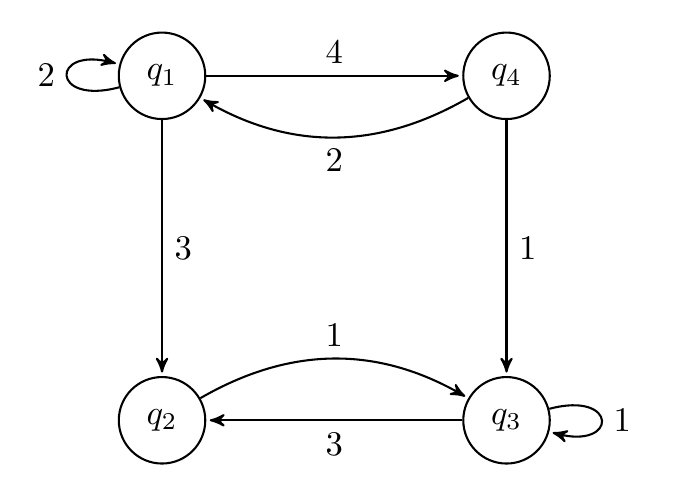}
    \caption{DFA $\mathcal{M}$ in Example~\ref{exp:2}.}
    \label{fig:automaton2}
\end{figure}

By choosing $h=1$ and $2d=2$, we run Algorithm \ref{algorithm2} 100 times by randomly  picking a polynomial $p_{0}(x) \in \operatorname{int}(\Sigma_{2 \mathrm{d}})$. Algorithm \ref{algorithm2}  \emph{always} finds the following cycle for the lifted system $S(\mathcal{A}_{\mathcal{M}})$:
\begin{equation}\label{cycle2}
3, 1, 1, 1
\end{equation}
whose 4th root of the corresponding spectral radius is equal to $0.84135421$, the lower bound shown in \eqref{equ:jsr_bounds2}. Note that the cycle \eqref{cycle2} is accepted by $L(\mathcal{M})$. Therefore, the asymptotic growth rate $\lim_{k \rightarrow \infty}\|(A_{3} A_{1}^3)^k\|_{2}^{\frac{1}{4k}}$ is equal to $0.84135421$.

As a comparison, we run the algorithm in \cite{legat2020certifying} 100 times to produce a sequence of matrices with a high asymptotic growth rate for the same $S(\mathcal{A},\mathcal{M})$. The success rate that it finds the same growth rate sequence as (\ref{cycle2}) remains around $50\%$ even if we we increase $2d$, the degree of the homogeneous polynomials, to 20 and increase $h$, the algorithm ``horizon'',  to 5. \hfill$\Box$

\end{example}

\begin{example}\label{exp:3}
Consider another constrained switching system $S(\mathcal{A},\mathcal{M})$ where the set $\mathcal{A}=\{A_1,A_2,A_3,A_4\}$ is 
$$
\begin{array}{ll}
A_{1}=\left(\begin{array}{cc}
\ 0.87\ & 0.48 \\
\ -0.79\ & -0.31
\end{array}\right), A_{2}=\left(\begin{array}{cc}
-0.29 &\ 0.45\ \\
-0.34 &\ -0.64\
\end{array}\right), \\
A_{3}=\left(\begin{array}{cc}
-0.77 & -0.88 \\
0.90 & 0.21
\end{array}\right), A_{4}=\left(\begin{array}{cc}
-0.79 & 0.38 \\
-0.67 & -0.94
\end{array}\right),
\end{array}
$$
and the DFA is $\mathcal{M}=(Q,U,f)$ with $Q=\{q_1,q_2,q_3,q_4\}$, $U=\{1,2,3,4\}$ and its transition map as shown in Fig~\ref{fig:automaton3}. 
Using the JSR toolbox in \cite{vankeerberghen2014jsr}, we obtain the following bounds for $\rho(\mathcal{A}, \mathcal{M})$: 
\begin{equation} \label{equ:jsr_bounds3}
1.03337866 \leq \rho(\mathcal{A},\mathcal{M}) \leq 1.03337876.
\end{equation}

\begin{figure}[ht]
    \centering
    \includegraphics[width=0.40\textwidth]{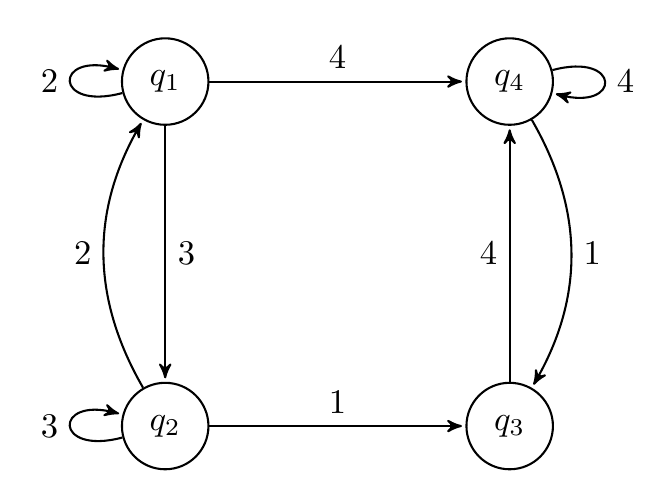}
    \caption{DFA $\mathcal{M}$ in Example~\ref{exp:3}.}
    \label{fig:automaton3}
\end{figure}

To evaluate the performance of  Algorithm \ref{algorithm2} proposed in this work and the algorithm in \cite{legat2020certifying}, 
we run both algorithms with various $d$ and $h$ for 100 times separately. 
The $\mathcal{M}$-accepted cycle with the largest average spectral radius found among all the tests is  
\begin{align}
    4,1,4\label{cycleex4}
\end{align}
whose 3rd root of the corresponding spectral radius is $1.03337866$, the lower bound shown in \eqref{equ:jsr_bounds3}. The success rate that our Algorithm \ref{algorithm2} generates cycle \eqref{cycleex4} is $100\%$ for all the cases. However, the algorithm in \cite{legat2020certifying} finds the same cycle with around $50\%$ success rate even when $2d = 20, h = 5$ . The average computation time for each run of the proposed Algorithm \ref{algorithm2} and the algorithm in \cite{legat2020certifying} are comparable; however, it is evident that our algorithm has a much higher success rate in generating $\mathcal{M}$-accepted  cycles. \hfill$\Box$
\end{example}



Table \ref{table1} summarizes the success rate and computation time of different algorithms to generate high asymptotic growth rate sequences for the same constrained switching systems. Comparing with the algorithm in \cite{legat2020certifying}, it's clear that Algorithm 1 in this paper has higher success rates to find the $\mathcal{M}$-accepted switching cycles for both example $\ref{exp:2}$ and example $\ref{exp:3}$ in which the constraining DFAs are not strongly connected and both can be separated into two strongly connected components. The $50\%$ success rate of the algorithm in \cite{legat2020certifying} comes from the fact that it randomly selects an initial state which has half chance to be in either one of the two strongly connected components. Our algorithm, however, generates the high-growth sequence on the lifted arbitrary system $S(\mathcal{A}_\mathcal{M})$ which can be seen as a particular case when the automaton has only one state and $m$ self-loops. Therefore, it has a better result when dealing with the switching system whose constraining DFA is not strongly connected, especially when there are a large number of strongly connected components.

\setlength{\topmargin}{-17pt}


\begin{table*}[t]
    \begin{center}
    \begin{tabular}{|c|c|c|c|cccc|}
    \hline Example & cycle length & Gripenberg Alg + $S(\mathcal{A}_{\mathcal{M}})$ & Conitope Alg + $S(\mathcal{A}_{\mathcal{M}})$ & $d$ & $h$ & \  Algorithm \ref{algorithm2} & \cite{legat2020certifying}\\
    \hline 2 & 8 & \makecell*[c]{100\%\\ (0.146s)} & \makecell*[c]{100\%\\ (12.2s)} & 1 & 3 & \makecell*[c]{92\%\\ (0.152s)} & \makecell*[c]{100\%\\ (0.084s)} \\
    & & & & 2 & 3 & \makecell*[c]{97\%\\ (0.543s)} & \makecell*[c]{100\%\\ (0.090s)} \\
    \hline 3 & 4 & \makecell*[c]{100\%\\ (0.025s)} & \makecell*[c]{100\%\\ (2.277s)} & 1 & 1 & \makecell*[c]{100\%\\ (0.143s)} & \makecell*[c]{52\%\\ (0.094s)} \\
    & & & & 2 & 3 & \makecell*[c]{100\%\\ (0.536s)} & \makecell*[c]{54\%\\ (0.096s)} \\
    & & & & 10 & 5 &  & \makecell*[c]{49\%\\ (0.085s)} \\
    \hline  
    4 & 3 & \makecell*[c]{100\%\\ (0.098s)} & \makecell*[c]{100\%\\ (1.637s)} & 1 & 2 & \makecell*[c]{100\%\\ (0.146s)} & \makecell*[c]{47\%\\ (0.110s)} \\
    & & & & 2 & 3 & \makecell*[c]{100\%\\ (0.537s)} & \makecell*[c]{55\%\\ (0.101s)} \\
     &   &   &   & 10 & 5 & & \makecell*[c]{50\%\\ (0.117s)} \\
    \hline
    \end{tabular}
    \end{center}
    \caption{Performance comparison of different algorithms where each algorithm is run 100 times for each example. The percentages show the success rate of each algorithm finding the highest growth rate sequences, and the average computation times (on a computer with 3.7 GHz CPU and 32GB memory) are included in the parentheses under each percentage. The second column shows the length of the highest growth rate sequences found for each example. The third and fourth columns show the success rates by applying the Gripenberg Algorithm \cite{gripenberg1996computing} and the Conitope Algorithm \cite{jungers2014lifted} on our lifted arbitrary switching system $S(\mathcal{A}_{\mathcal{M}})$ respectively. The last column compares the success rate and computation time of Algorithm 1 proposed in this work and the algorithm in \cite{legat2020certifying}.}
    \label{table1}
\end{table*}

\section{Conclusion}\label{sec:conclusion}

In this paper, we proposed a novel algorithm, based on the matrix-form and lift-based expression of the constrained switching system $S(\mathcal{A},\mathcal{M})$ and the dual solution of the SOS approximation, to generate a sequence of matrices with a high asymptotic growth rate. We proved that the high asymptotic growth rate sequence generated by our algorithm for the lifted arbitrary switching system $S(\mathcal{A}_{\mathcal{M}})$ is an equivalently satisfying sequence for the original constrained switching system $S(\mathcal{A},{\mathcal{M}})$. We showed that the proposed lift-based method is applicable to a type of existing algorithms. We also showed that compared with some existing algorithms, the lift-based algorithm has a higher success rate of generating a sequence of matrices with a large asymptotic growth rate in several cases. In future work, we plan to develop more efficient algorithms on the closeness of the asymptotic growth rate to $\rho(\mathcal{A},\mathcal{M})$ by incorporating the proposed lifting method and other existing JSR approximation algorithms.


\bibliographystyle{IEEEtran}
\bibliography{./CDC21}
\end{document}